\documentclass{article}
\usepackage{amssymb}
\usepackage{amsmath}
\usepackage{amsfonts}
\usepackage{graphicx}

\newtheorem{theorem}{Theorem}[section]
\newtheorem{lemma}[theorem]{Lemma}

\newtheorem{remark}[theorem]{Remark}

\newcommand{\qed}{\hfill \rule{2.3mm}{2.3mm}}

\newcommand{\om}{\omega}
\newcommand{\al}{\alpha}

\newcommand{\la}{\lambda}

\newcommand{\ve}{\varepsilon}
\newcommand{\vp}{\varphi}

\newcommand*{\esssup}{\operatorname*{ess\phantom{|}\!sup}}
\newcommand*{\essinf}{\operatorname*{ess\phantom{|}\!inf}}

\newcommand{\bee}{\begin{equation}}
\newcommand{\ee}{\end{equation}}
\newcommand{\uu}{\bar{u}_\ve}

\newcommand{\R}{\mathbb{R}}
\newcommand{\N}{\mathbb{N}}

\newcommand{\Z}{\mathbb{Z}}

\newcommand{\x}{x/\varepsilon}

\newcommand{\LL}{{\cal L}}
\newcommand{\reff}[1]{(\ref{#1})} 
  
\begin{document}

\title{\bf
An H-convergence-based implicit function theorem for homogenization of nonlinear non-smooth elliptic systems 
}

\newcounter{thesame}
\setcounter{thesame}{1}
\author{
Lutz Recke \thanks{Humboldt University of Berlin, Institute of Mathematics, Rudower Chaussee 25, 12489 Berlin, Germany.
		{\small   E-mail:
			{\tt lutz.recke@hu-berlin.de}}
}}
\date{}

\maketitle

\begin{abstract}
\noindent
We consider  homogenization of semilinear 
elliptic PDE systems of the type
$$
\partial_{x_i}\Big(a^{\al \beta}_{ij}(\ve,x)
\partial_{x_j}u^\beta(x)
+f_i^\al(x,u(x))\Big)=0 \mbox{ in }  \Omega, \; \al=1,\ldots,n,
$$
with homogeneous Dirichlet boundary conditions.
Here $\ve>0$ is the small homogenization parameter, $\Omega \subset \R^N$ is a bounded Lipschitz domain, 
$a^{\al \beta}_{ij}(\ve,\cdot)\in L^\infty(\Omega)$ satisfy the Legendre ellipticity condition,
and the maps $u \in C(\overline\Omega;\R^n) \mapsto
f_i^\al(\cdot,u(\cdot)) \in L^{p_0}(\Omega)$ are $C^1$-smooth with certain $p_0>N$.
We suppose that the family of diffusion tensor functions $[a^{\al \beta}_{ij}(\ve,\cdot)]$ H-converges for $\ve \to 0$
and that
$$
\mbox{$N=2$ or 
$a^{\al \beta}_{ij}=0$ for $\al>\beta$.}
$$

Our result is of implicit function theorem type: If $u_0$ is a non-degenerate weak solution to the homogenized problem, then 
for $\ve \approx 0$ there exists exactly one weak solution $u=u_\ve$
with $\|u-u_0\|_\infty \approx 0$, and $\|u_\ve-u_0\|_\infty\to 0$ for $\ve \to 0$.

The main tools of the proofs are
gradient estimates of Meyers and  Morrey type for solutions to linear elliptic systems with non-smooth data.
Neither assumptions about global solution  uniqueness are needed
nor additional smoothness of $\partial \Omega$ or $a_{ij}^{\al \beta}$ or $f_i^\beta$ or $u_0$ 
nor growth restrictions 
%of the space dimension $N$ or 
for  $f_i^\al(x,\cdot)$.
\end{abstract}

%\begin{keyword} 


\section{Introduction and main results}
\setcounter{equation}{0}
\setcounter{theorem}{0}

We consider homogenization of
boundary value problems for
semilinear second-order elliptic systems of the type
\bee
\label{BVP}
\left.
\begin{array}{l}
\partial_{x_i}\Big(a^{\al \beta}_{ij}(\ve,x)
\partial_{x_j}u^\beta(x)
+f_i^\al(x,u(x))\Big)=0 \mbox{ in }  \Omega,\\
u^\al(x)=0\mbox{ on } \partial\Omega
\end{array}
\right\}
\al=1,\ldots,n.
\ee
Here $\ve>0$ is the small homogenization parameter, and $n \in \N$ is the system dimension.
In \reff{BVP} and in what follows  repeated indices are to be summed over 
 $\al,\beta,\ldots=1,\ldots,n$ and
$i,j,\ldots=1,\ldots,N$, where $N \in \N$ is the space dimension.
 We suppose that
\begin{eqnarray}
\label{Omass}
&&\mbox{$\Omega$ is a bounded Lipschitz domain in $\R^N$,}\\
\label{aass}
&&a_{ij}^{\al \beta}(\ve,\cdot) \in L^\infty(\Omega), \; \esssup\left\{\left|a_{ij}^{\al \beta}(\ve,x)\right|:\, \ve>0, \,x \in \Omega\right\}<\infty,
\\
\label{mona}
&&\essinf\left\{a^{\al \beta}_{ij}(\ve,x)v^\al_i v^\beta_j:\; \ve>0,\, x\in \Omega,\,v \in \R^{nN},\; v^\al_iv^\al_i=1\right\}>0,\\
\label{cass}
&&u \in C(\overline\Omega;\R^n) \mapsto 
f_i^{\al}(\cdot,u) \in L^{p_0}(\Omega) \mbox{ are $C^1$-smooth with certain } 
p_0>N.
\end{eqnarray}
More in detail formulated, assumption \reff{cass} means that there exists $p_0>N$ such that for all $u \in 
C(\overline\Omega;\R^n)$ the functions 
$f_i^{\al}(\cdot,u(\cdot))$ belong to  $L^{p_0}(\Omega)$, and that the maps
$u \mapsto f_i^{\al}(\cdot,u(\cdot))$
are $C^1$-smooth from the function space $C(\overline\Omega;\R^n)$ into the function space $L^{p_0}(\Omega)$.
This is satisfied if, for example, 
the functions $f_i^\al$ are products of the type
$
f_i^\al(x,u)=g_{i}^{\al}(x)h_i^\al(u)$
%\mbox{ 
(no summation over $\al$ or $i$)
with $g_i^\al\in L^{p_0}(\Omega)$ and 
$h_i^\al \in C^1(\R^n)$
or if $f_i^\al$ are  finite sums of such products. More general sufficient conditions for \reff{cass} are stated below in Appendix 2 of the present paper.

Further, we suppose that the family of diffusion tensor functions $[a^{\al \beta}_{ij}(\ve,\cdot)]$ H-converges for $\ve \to 0$ to a tensor function $[\hat a^{\al \beta}_{ij}]$
with $\hat a^{\al \beta}_{ij} \in L^\infty(\Omega)$, which means the following:
\bee
\label{H}
\left.
\begin{array}{l}
\mbox{For any $\phi \in W^{-1,2}(\Omega;\R^n)$
and for the corresponding $u_\ve,\hat u \in W_0^{1,2}(\Omega;\R^n)$,}\\
\mbox{defined by }
\partial_{x_i}(a^{\al \beta}_{ij}(\ve,\cdot)
\partial_{x_j}u_\ve^\beta)=
\partial_{x_i}(\hat a^{\al \beta}_{ij}
\partial_{x_j}\hat u^\beta)=\phi^\al, 
\mbox{ one has for $\ve \to 0$ that}\\
\mbox{$u^\al_\ve \rightharpoonup \hat u^\al$
weakly in $W_0^{1,2}(\Omega)$
and $a^{\al \beta}_{ij}(\ve,\cdot)
\partial_{x_j}u_\ve^\beta
\rightharpoonup
\hat a^{\al \beta}_{ij}
\partial_{x_j}\hat u^\beta$
weakly in $L^{2}(\Omega)$.}
\end{array}
\right\}
\ee
Finally, we suppose that 
\bee
\label{zwei}
N=2 
\ee
or 
\bee
\label{gzwei}
a_{ij}^{\al \beta}(\ve,x)=0 \mbox{ for   all } \al>\beta  \mbox{ and } \ve >0 \mbox{ and for almost all } x \in \Omega.
\ee

A vector function $u \in W_0^{1,2}(\Omega;\R^n)\cap C(\overline\Omega;\R^n)
$ is called weak solution to the boundary value problem~\reff{BVP} if it satisfies the variational equation
$$
\int_\Omega\Big(
a^{\al \beta}_{ij}(\ve,x)
\partial_{x_j}u^\beta(x)
+f_i^\al(x,u(x))\Big)\partial_{x_i}\vp^\al(x)dx=0
\mbox{ for all } \vp \in W_0^{1,2}(\Omega;\R^n),
$$
and similarly for the homogenized boundary value problem
\bee
\label{hombvp}
\left.
\begin{array}{l}
\partial_{x_i}\Big(\hat{a}^{\al \beta}_{ij}(x)
\partial_{x_j}u^\beta(x)+
f_i^\al(x,u(x))\Big)
=0 \mbox{ in } \Omega,\\
u^\al(x)=0 \mbox{ on } \partial\Omega
\end{array}
\right\}
\al=1,\ldots,n
\ee
and for its linearization \reff{linhombvp}. 
Also, we denote by
$
\|u\|_\infty:=\sum_{\al=1}^n \esssup_{x \in \Omega} |u^\al(x)|
$
the norm in the Lebesgue space $L^{\infty}(\Omega;\R^n)$.

Our main result is the following
\begin{theorem} 
\label{main}
Suppose \reff{Omass}-\reff{H}.
Further, suppose that \reff{zwei} or that \reff{gzwei} is satisfied.
Finally, suppose that there is given a weak solution
$u=u_0 \in W_0^{1,2}(\Omega;\R^n)
\cap C(\overline\Omega;\R^n)$
to \reff{hombvp} such that 
the linearized homogenized  boundary value problem
\bee
\label{linhombvp}
\left.
\begin{array}{l}
\partial_{x_i}\Big(\hat{a}^{\al \beta}_{ij}(x)
\partial_{x_j}u^\beta(x)+\partial_{u^\beta}f_i^\al(x,u_0(x))u^\beta(x)\Big)
=0 \mbox{ in } \Omega,\\
u^\al(x)=0 \mbox{ on } \partial\Omega
\end{array}
\right\}
\al=1,\ldots,n
\ee
does not have nontrivial weak solutions $u \in W_0^{1,2}(\Omega;\R^n)\cap C(\overline\Omega;\R^n)$.

Then there exist $\ve_0>0$ and $\delta>0$ such that for all $\ve \in (0,\ve_0]$ there exists exactly one weak solution
$u=u_\ve \in W_0^{1,2}(\Omega;\R^n)
\cap C(\overline\Omega;\R^n)$ to \reff{BVP} with $\|u-u_0\|_\infty <\delta$.
Moreover,
%\label{infest}
$\|u_\ve-u_0\|_\infty \to 0$ for $\ve \to 0$.
\end{theorem}

The idea of the proof of Theorem \ref{main} is rather simple: We use, similarly to the proof of the classical implicit function theorem, 
%Banach's fixed point theorem for the 
the sequence $u_1,u_2,\ldots \in W_0^{1,2}(\Omega;\R^n)
\cap C(\overline\Omega;\R^n)$, which 
is  defined by
$$
%\left.
\begin{array}{l}
\partial_{x_i}\Big(a^{\al \beta}_{ij}(\ve,x)\partial_{x_j}u_{l+1}^\beta(x)
+\partial_{u^\beta}f_i^\al(x,u_0(x))
\left(u_{l+1}^\beta(x)-u_l^\beta(x)\right)+f_i^\al(x,u_l(x))\Big)=0 \mbox{ in } \Omega,\\
u^\al_{l+1}(x)=0 \mbox{ on } \partial\Omega
\end{array}
%\right\}
%\al=1,\ldots,n.
$$
for $\al=1,\ldots,n$.
But it is not simple to answer the question in which function space (with respect to which norm) this sequence converges.
We take Sobolev spaces $W^{1,p}(\Omega;\R^n)$ with $p>2$, but $p\approx 2$, in the case \reff{zwei} and  Sobolev-Morrey spaces $W^{1,2,\la}(\Omega;\R^n)$ with $\la>N-2$, but $\la \approx N-2$, in the case \reff{gzwei}.
Also, it is not simple to answer the question what should be an appropriate starting element $u_1$ of the sequence. It turns out that $u_0$ is not an appropriate
starting element.
We take an $\ve$-depending $u_1$ defined by
\bee
\label{initial}
\left.
\begin{array}{l}
\partial_{x_i}\Big(a^{\al \beta}_{ij}(\ve,x)
\partial_{x_j}u_{1}^\beta(x)
+f_i^\al(x,u_0(x))\Big)=0 \mbox{ in } \Omega,\\
u_1^\al(x)=0 \mbox{ on } \partial \Omega,
\end{array}
\right\} 
\al=1,\ldots,n.
\ee
Remark that neither the Sobolev spaces nor the Sobolev-Morrey spaces nor this starting element appear in the formulation of Theorem \ref{main}, they are hidden only.

This approach is well-known for singularly perturbed nonlinear ODEs and elliptic and parabolic PDEs, see \cite{Butetc,But2022,Fiedler,
Magnus1,Magnus2,NURS,OmelchenkoRecke2015,Recke2022,
ReckeOmelchenko2008}. It is based on the generalized implicit function theorem
of R.J. Magnus \cite[Theorem 1.2]{Magnus1} and on several of its modifications. Hence, these 
generalized implicit function theorems permit a common approach to nonlinear singular perturbation and  homogenization.

Roughly speaking, the reason why the classical implicit function theorem cannot be applied
to~\reff{BVP} is the following:
The linearized operators  converge for $\ve \to 0$ 
in a very weak sense only
(in the sense of H-convergence), 
not with respect to a uniform operator norm.  
Remark that in the classical implicit function theorem one cannot omit, in general, the assumption, that the linearized operators converge for $\ve \to 0$ with respect to a uniform operator norm (cf. \cite[Section~3.6]{Katz}).

\begin{remark}
The concept of H-convergence was introduced by F. Murat and L. Tartar (see, e.g. \cite{Tartar} or \cite[Section 13]{Ci}).
Assumption \reff{H} is satisfied, for example, for periodic homogenization, i.e. 
$$
a^{\al \beta}_{ij}(\ve,x)
=\tilde a^{\al \beta}_{ij}(\x)
\mbox{ with $\Z^N$-periodic $
\tilde a^{\al \beta}_{ij}
\in L^\infty(\R^N)$}
$$ 
or for periodic homogenization with localized defects, i.e.
$
a^{\al \beta}_{ij}(\ve,x)
=\tilde a^{\al \beta}_{ij}(\x)+
b^{\al \beta}_{ij}(\x)$
with
$\Z^N$-periodic $\tilde a^{\al \beta}_{ij}
\in L^\infty(\R^N)$ and
$b^{\al \beta}_{ij} \in L^\infty(\R^N)$ and
$$
\lim_{r \to 0}\frac{1}{r^N}\int_{\|y-x\|<r} |b^{\al \beta}_{ij}(y)|dy=0
\mbox{ for almost all } x \in \R^N,
$$
see \cite{Blanc,BDL}. In these cases the homogenized diffusion tensor is $x$-independent and can be represented
rather explicit as
$$
\hat{a}^{\al\beta}_{ij}=
\int_{(0,1)^N}\left(\tilde a^{\al \beta}_{ij}(y)
+\tilde a^{\al \gamma}_{ik}(y)
\partial_{y_k}v^{\gamma \beta}_j(y)\right)dy,
$$
where the $nN$ correctors $v^{\beta}_j=(v^{1 \beta}_j,\ldots,v^{n \beta}_j)
\in W^{1,2}_{\rm loc}(\R^N;\R^n)$, $\beta=1,\ldots,n$, $j=1,\ldots,N$,
are are the weak solutions to
the $nN$ cell problems
$$
\left.
\begin{array}{l}
\partial_{y_i}
\left(\tilde a^{\al \beta}_{ij}(y)+
\tilde a^{\al \gamma}_{ik}(y)
\partial_{y_k}v^{\gamma \beta}_j(y)\right)
=0\mbox{ for } y\in \R^N,\\
v^{\al \beta}_j(\cdot+z)=v^{\al \beta}_j \mbox{ for } z \in \Z^N,\;
\displaystyle\int_{(0,1)^N} v^{\al \beta}_j(y)dy=0
\end{array}
\right\}
\al=1,\ldots,n.
$$
In the case of periodic homogenization with localized defects and with one space dimension, i.e. $N=1$, correctors are not needed, in this case the representation is
$$
%\label{N1}
[\hat a^{\al \beta}]=\left(\int_0^1 
[\tilde a^{\al \beta}(y)]^{-1}dy\right)^{-1}.
%\mbox{ if } a^{\al \beta}(\ve,x)
%=\tilde a^{\al \beta}(\x)
%\mbox{ and } \tilde a^{\al \beta}(\cdot+1)=
%a^{\al \beta}.
%A(y):=[a^{\al \beta}(y)] \in M_n.
$$
For this particular case Theorem \ref{main} has been proved in \cite{ReckeNonper}.

For H-convergence of families of diffusion tensor functions describing laminated or perforated materials see, e.g. \cite[Chapters 12, 15 and 16]{Tartar}.

%Finally, let us remark that \reff{N1} shows that the map 
%$[a_{ij}^{\al \beta}]\mapsto
%[\hat a_{ij}^{\al \beta}]$ is homogeneous, but not additive, in general.
\end{remark}

\begin{remark}
We do not believe that 
%in the case of
%$N>2$ space dimensions 
the assertions of
Theorem \ref{main}
remain true if neither \reff{zwei} nor \reff{gzwei} are satified, in general.
%for   general elliptic systems (i.e. for systems with  system dimension $n>1$ and with diffusion tensors which are far from being triangular), no matter if localized defects are present or not.
The reasons for this guess are the well-known examples of unbounded weak solutions to linear elliptic systems 
with  $L^\infty$-coefficients and
with system dimension $n>1$, space dimension 
$N>2$ and with diffusion tensors which are far from being triangular (cf., e.g. \cite[Section 8.7]{BenF},
\cite[Section 12.2]{Chen},
\cite[Section 6.2]{Giusti}).

Let us remark that condition \reff{gzwei} is satisfied for systems with
$a_{ij}^{\al \beta}=0$ for $\al\not=\beta$ (which often are  called
weakly coupled systems or systems with diagonal main part or systems without cross-diffusion).
Especially, condition \reff{gzwei} is satisfied in the case $n=1$
(scalar elliptic equations).
\end{remark}

\begin{remark}
It turns out that Theorem \ref{main} is new even in the linear case, i.e. if
$$
f_i^\al(x,u)=g_i^\al(x)
\mbox{ for all } u \in \R^n
\mbox{ and almost all } x \in \Omega, \mbox{ where } g_i^\al \in L^{p_0}(\Omega) \mbox{ with } p_0>N.
$$
In this case the boundary value problems 
\reff{BVP} and \reff{hombvp}
are uniquely weakly solvable (because of the Lax-Milgram lemma),
but it is not clear if the weak solutions belong to $L^\infty(\Omega;\R^n)$ and if 
they satisfy
$\|u_\ve-u_0\|_\infty \to 0$ for $\ve \to 0$.
\end{remark}

\begin{remark}
Theorem \ref{main} is true also for mixed boundary conditions of the type
$$
\left.
\begin{array}{l}
\Big(a^{\al \beta}_{ij}(\ve,x)\partial_{x_j}u^\beta(x)
+f_i^\al(x,u(x))\Big)\nu_i(x)=f^\al(x,u(x))
\mbox{ on } \Gamma,\\
u^\al(x)=0  \mbox{ on } \partial \Omega \setminus \Gamma,
\end{array}
\right\}
\al=1,\ldots,n
$$
with $C^1$-maps $u \in C(\overline\Omega;\R^n) \mapsto f^{\al}(\cdot,u(\cdot)) \in L^{p_0-1}(\Gamma)$,
where 
$\Gamma\subseteq \partial\Omega$, 
$\nu:\partial\Omega \to \R^N$ is the unit outer normal vector field on $\partial\Omega$,  
and where the set $\Omega \cup \Gamma$ 
is regular in the sense of
\cite[Definition~2]{G}.
The reason for that is that
the regularity results Theorem \ref{Maxreg1} and 
Theorem \ref{Maxreg2}
(and, hence, the Theorems \ref{maxreg1} and 
\ref{maxreg2})
below are true for those boundary conditions also.

Remark that the class of bounded open sets $\Omega \subset \R^N$ which are  
regular the sense of
\cite[Definition~2]{G} (with $\Gamma=\emptyset$) is sligtly larger than the class of bounded  Lipschitz domains. In particular, bi-Lipschitz transformations of regular sets are regular again.
Hence, Theorem \ref{main} is true not only for bounded Lipschitz domains
$\Omega\subset \R^N$, but also for bounded open sets $\Omega\subset \R^N$ which are regular in the sense of \cite[Definition 2]{G}.

In \cite{II} is presented a result
of the type of Theorem \ref{main} 
for periodic homogenizaton of 2D semilinear 
elliptic systems with non-smooth data and with various (including mixed) boundary conditions.

In \cite{GrR,Recke1995} are proven results of implicit function theorem type for 
quasilinear elliptic systems with non-smooth data and mixed boundary conditions, but without highly oscillating coefficients.
\end{remark}

\begin{remark}
Theorem \ref{main} is true also for for elliptic systems with lower order terms of the type
$$
\partial_{x_i}\Big(a^{\al \beta}_{ij}(\ve,x)
\partial_{x_j}u^\beta(x)
+f_i^\al(x,u(x))\Big)=f_j^{\al \beta}(x,u(x))
\partial_{x_j}u^\beta(x)+f^{\al}(x,u(x))
$$
with $C^1$-maps $u \in C(\overline\Omega;\R^n) \mapsto f_j^{\al \beta}(\cdot,u(\cdot)) \in L^{\infty}(\Omega)$ and 
 $u \in C(\overline\Omega;\R^n) \mapsto f^\al(\cdot,u(\cdot)) \in L^{p_0/2}(\Omega)$.
\end{remark}

%\begin{remark} In the case \reff{gzwei} the assertions of Theorem \ref{main} remain true if assumption $a_{ij}^{\al \beta}=0$ for $\al>\beta$ is weakened to $\|a_{ij}^{\al \beta}(\ve,\cdot)\|_\infty\approx 0$ for $\al>\beta$ and $\ve>0$. This follows because Theorem \ref{maxreg2} below is true also for  elliptic systems which are close to be triangular. How small $\|a_{ij}^{\al \beta}(\ve,\cdot)\|_\infty$ for $\al>\beta$ should be, this depends on
%$\Omega$, $\|a_{ij}^{\al \beta}(\ve.\cdot)\|_\infty$ for $\al\le\beta$
%and on the essential infimum in \reff{mona}.
%\end{remark}

\begin{remark}
In fact,  below we show slightly more than claimed in 
Theorem \ref{main}, namly that
the solutions 
$u_\ve$ to \reff{BVP}
%and $u_0$ 
are H\"older continuous.
%and the corresponding H\"older norm of $u_\ve-u_0$ tends to zero
%for $\ve \to 0$. 
\end{remark}

%\begin{remark}
%Assumption \reff{cass} is satisfied, for example, if the functions $f_i^\al$ are products of the type
%$$
%f_i^\al(x,u)=g_{i}^{\al}(x)h_i^\al(u)
%\mbox{ (no summation over $\al$ and $i$)}
%$$
%with $g_i^\al\in L^{p_0}(\Omega)$ and 
%$h_i^\al \in C^1(\R^n)$
%or if $f_i^\al$ are  finite sums of such products.
%\end{remark}

\begin{remark}
It is well-known that assumption \reff{H} implies that the homogenized diffusion coefficient functions $\hat{a}^{\al \beta}_{ij}$ 
satisfy the Legendre condition
\reff{mona} also, i.e.
\bee
\label{monAhat}
\essinf\left\{\hat a^{\al \beta}_{ij}(x)v^\al_i v^\beta_j:\; x \in \Omega,\,v \in \R^{nN},\; v^\al_iv^\al_i=1\right\}>0.
\ee
Therefore,
the rather implicit assumption of Theorem \ref{main}, that there do not exist weak solutions $u\not=0$
to \reff{linhombvp}, is satisfied, for example,  
if $\|\partial_{u^\beta}f^\al_i(\cdot,u_0(\cdot))\|_\infty$ are sufficiently small.
But smallness of $\|\partial_{u^\beta}f^\al_i(\cdot,u_0(\cdot))\|_\infty$
is far from being necessary for that assumption.

But remark that, if the diffusion coefficients $a_{ij}^{\al \beta}$ satify the condition \reff{gzwei}, then the homogenized diffusion coefficients $\hat a_{ij}^{\al \beta}$ do not, in general.
\end{remark}

\begin{remark}
Because of the non-smoothness of the data of \reff{BVP} we are not able to estimate the rate of the convergence  $\|u_\ve-u_0\|_{\infty}\to 0$ for $\ve \to 0$.

For periodic homogenization 
for linear  elliptic systems with
smooth data and Dirichlet boundary conditions  
it can be shown by means of
Green function estimates
that
\bee
\label{Pest1}
\|u_\ve-u_0\|_\infty=O(\ve)  
\mbox{ for } \ve \to 0
\ee
(cf. \cite[Theorem 3.4]{Kenig}).
The same is true for 
linear scalar elliptic equations with smooth data and Dirichlet boundary conditions, where in the proofs the maximum principle is used
(cf. \cite[Section~2.4]{Ben}).
We do not believe that \reff{Pest1} is true for problems of the type \reff{BVP}, in general, even
in the case of periodic homogenization and of linear equations.

In \cite{II} is considered periodic homogenization for semilinear elliptic systems with non-smooth data of the type \reff{BVP} with $N=2$
and with $u_0 \in W_0^{2,p}(\Omega;\R^n)$ with $p>2$, 
and it is shown that
$$
\|u_\ve-u_0\|_\infty=O(\ve^\om)  
\mbox{ for } \ve \to 0
\mbox{ for all $\om\in [0,1/2)$}.
$$
Sinilarly, in \cite{III} is considered periodic homogenization for semilinear elliptic scalar equations with non-smooth data of the type \reff{BVP} with $n=1$
and with $u_0 \in W_0^{2,p}(\Omega)$ with $p>N$, 
and it is shown that
$$
\|u_\ve-u_0\|_\infty=O(\ve^\om)  
\mbox{ for } \ve \to 0
\mbox{ for all $\om\in [0,1/N)$}.
$$

For periodic homogenization of linear elliptic problems with smooth data 
also rates of 
$W^{1,2}$-estimates 
of $u_\ve-\uu$
are known, where $\uu$ is an appropriate  family of approximate solutions, defined by $u_0$ and  correctors, see, e.g.  
%\cite{Pakhnin} 
%homogeneous Dirichlet problems for elliptic systems of the type \reff{BVP}  with $\partial \Omega \in C^2$, $b_{ij}^{\al \beta}=c_i^{\al,\beta}=0$ and $f_i^\al \in W^{1,2}(\Omega)$ are considered, and it is shown that
%\bee
%\label{onehalf}
%\|u_\ve-\uu\|_{1,2}=O(\ve^{1/2}) \mbox{ for } \ve \to 0
%\ee
\cite[Theorem 14.3]{Ci1},
\cite[Theorems 3.2.2 and 3.2.3]{Shen},
\cite[Theorem 3.2]{Xu}) and
\cite{O1,O2,Pakhnin,Zhuge}.
\end{remark}

\begin{remark}
There exists an astonishing analogy between  approaches to singularly perturbed problems and to homogenization problems (cf. \cite{I,II}). The analogy  consists in
the use of so-called approximate solutions, i.e. of families (with family parameter $\ve>0$)
of functions which satisfy the problem (i.e.  differential equations,  boundary conditions etc.) approximately for small singular perturbation parameter $\ve$ or for small homogenization parameter $\ve$, respectively.
Usually those families of approximate solutions are "constructed" by means of ansatzes, streched variables, correctors, formal asymptotic expansions, smoothing operators, cut-off functions etc., and they are choosen according of the requirements of the problem.
Often it is not clear at the very beginning, if for $\ve \to 0$ those approximate solutions approximate (and in which sense) exact solutions of the given problem or not.

In the present paper we use the family of approximate solutions $\uu$, which is  defined in \reff{barudef1} or, what is the same, by $\uu:=u_1$ with $u_1$ from 
\reff{initial}.
The advantage of this choice is that 
no additional regularity assumptions for $u_0$ and, hence, no smoothing operators are needed. This is new and allows to work in Morrey spaces (where smooth functions are not dense)
and, this way, to prove Theorem \ref{main} in the case \reff{gzwei}.

The disadvantage of these approximate solutions $\uu$ is that they
are defined as the solutions of the linear boundary value problems~\reff{initial}
and, hence, to determine $\uu$ numerically is not much simpler than to determine the exact solutions $u_\ve$ to \reff{BVP} numerically. Another disadvantage of $\uu$ is that they do not help to estimate the rate of convergence of $\|u_\ve-u_0\|_\infty \to 0$ for $\ve \to 0$, because Lemma \ref{barulemma} below is proved by assuming the contrary.

Other families of approximate solutions $\uu$, which are constructed by means of two-scale asymptotic expansions, the function $u_0$, correctors, cut-off functions and smoothing operators, are used, e.g. in \cite[Chapters 3.1 and 3.2]{Shen} or \cite{II,III}, and under additional smoothness assumptions they allow to estimate the  rate of convergence $\|u_\ve-\uu\|_\infty \to 0$ for $\ve \to 0$.
%like \reff{onehalf}.

For another analogy of techniques used in singular perturbation theory and in homogenization theory see Lemma \ref{Rob} below.
\end{remark}

\begin{remark}
\label{Neuk}
What concerns existence and local uniqueness
for periodic homogenization 
of semilinear elliptic PDEs
(without assumption of global uniqueness), besides \cite{II,III}
we know only the result 
\cite{Bun}  for scalar semilinear elliptic PDEs of the type
$
\nabla\cdot(a(\x) 
\nabla u(x))=f(x)g(u(x)),
$
where the nonlinearity $g$ is supposed to have a sufficiently small local Lipschitz constant (on an appropriate bounded interval). Let us mention also \cite{Lanza1,Lanza2,Riva}, where existence and local uniqueness for  periodic homogenization problems for the linear Poisson equation with highly oscillating nonlinear Robin boundary conditions is shown. There the specific structure of the problem (no highly oscillating diffusion coefficients) allows to apply the classical implicit function theorem.

Nonperiodic homogenization results for semilinear elliptic PDEs are presented also in \cite[Section 5.2]{N}. There the solutions are globally unique because of monotonicity assumptions with respect to the nonlinearities.
\end{remark}

Our paper is organized as follows:

In Section \ref{abset} we introduce the abstract equations \reff{abBVP} and \reff{abhomBVP}, which are equivalent to the weak settings of the boundary value problems \reff{BVP} and \reff{hombvp}, respectivly. 

In Section \ref{rift}
we formulate and prove an  abstract result of implicit function theorem type for equation \reff{abBVP} which will be used later in order to prove Theorem \ref{main}.
Roughly speaking, this abstract result claims the following:
If appropriate function spaces 
%of $W_0^{1,2}(\Omega;\R^n) \cap L^{\infty}(\Omega;\R^n)$
(cf. Subsection \ref{spaces})
exist, then one can use certain families of
appropriate  approximate solutions to \reff{BVP}
(cf. Subsection \ref{approximate}) 
in order to solve 
\reff{abBVP} locally uniquely by means of Banach's fixed point theorem 
(cf. Subsection \ref{Banach})
similarly to the proof of the classical implicit function theorem.
Here we use Theorem \ref{Shentheorem} about H-convergence of the diffusion tensors fields $[a_{ij}^{\al \beta}(\ve,\cdot)]$ for $\ve \to 0$.

In Sections \ref{choicezwei} and \ref{choicegzwei} we show that certain Sobolev spaces 
or certain Sobolev-Morrey spaces 
are appropriate in the cases \reff{zwei} 
or \reff{gzwei}, respectively, to apply Theorem \ref{Banach} to \reff{abBVP}.
Here we use Theorems 
\ref{Maxreg1} and \ref{Maxreg2} about maximal solution regularity for linear elliptic systems with non-smooth data.

In Appendix 1 we formulate two known results about maximal solution regularity for linear elliptic systems with non-smooth data: Theorem \ref{Maxreg1} concerns the pairs of Sobolev spaces $W_0^{1,p}(\Omega;\R^n)$ and $W^{-1,p}(\Omega;\R^n)$ with $p\approx 2$ in the case \reff{zwei}, and Theorem \ref{Maxreg2} concerns the pairs of Sobolev-Morrey spaces $W_0^{1,2,\la}(\Omega;\R^n)$ and $W^{-1,2,\la}(\Omega;\R^n)$ with $\la \approx N-2$ in the case \reff{gzwei}.

Finally, in Appendix 2 we formulate a more general sufficient condition concerning the functions $f_i^\al:\Omega\times \R^n \to \R^n$ such that the assumtion \reff{cass} is satisfied.
\\

In what follows we will prove Theorem \ref{main}. 
For that we use the objects of Theorem \ref{main}: The bounded Lipschitz domain $\Omega \subset \R^N$, the diffusion coefficients $a^{\al \beta}_{ij}(\ve,\cdot)$ with \reff{aass} and  \reff{mona}, 
%and with \reff{zwei} or \reff{gzwei},
the nonlinearities $f_i^\al$ with \reff{cass},
%the periodic correctors $v^{\beta}_j \in W^{1,2}_{\rm loc}(\R^N;\R^n)$, which are defined by the cell problems \reff{cell},
the homogenized diffusion coefficients
$\hat a^{\al \beta}_{ij}\in L^\infty(\Omega)$, which are
defined by \reff{H}, and the
nondegenerate
weak solution $u_0\in W_0^{1,2}(\Omega;\R^n)\cap C(\overline\Omega;\R^n)
$ to the homogenized boundary value problem~\reff{hombvp}.

\section{Abstract setting of the boundary value problems \reff{BVP}
and \reff{hombvp}}
\label{abset}
\setcounter{equation}{0}
\setcounter{theorem}{0}
As usual, the norm in the Sobolev space $W^{1,2}(\Omega;\R^n)$
is denoted by 
$$
\|u\|_{1,2}:=
\left(\sum_{\al=1}^n
\int_\Omega
\left(|u^\al(x)|^2+
\sum_{i=1}^N 
|\partial_{x_i}u^\al(x)|^2\right)
dx\right)^{1/2}.
$$
The subspace $W_0^{1,2}(\Omega;\R^n)$ 
of $W^{1,2}(\Omega;\R^n)$ 
is the closure with respect to this norm 
of the set 
of all $C^{\infty}$-maps $u:\Omega \to \R^n$ with compact support,
and $W^{-1,2}(\Omega;\R^n):=W_0^{1,2}(\Omega;\R^n)^*$ is the dual space to 
$W_0^{1,2}(\Omega;\R^n)$ with dual space norm
$$
\|\phi\|_{-1,2}:=\sup\left\{\langle \phi,\vp\rangle_{1,2}:\; \vp \in W_0^{1,2}(\Omega;\R^n),
\|\vp\|_{1,2}\le1\right\},
$$
where
$\langle \cdot,\cdot\rangle_{1,2}:
W^{-1,2}(\Omega;\R^n)\times
W_0^{1,2}(\Omega;\R^n)\to \R$
is the dual pairing. 

Let us  introduce linear bounded operators $\hat A,A_\ve:W^{1,2}(\Omega;\R^n)\to W^{-1,2}(\Omega;\R^n)$
(for $\ve>0$) 
and $D:L^{2}(\Omega;\R^{nN})\to W^{-1,2}(\Omega;\R^n)$
by
\bee
\label{ABdef}
\left.
\begin{array}{rcl}
\langle \hat A u,\vp\rangle_{1,2}&:=&
\displaystyle\int_\Omega \hat a_{ij}^{\al \beta}(x)
\partial_{x_j}u^\beta(x)\partial_{x_i}\vp^\al(x)dx,\\
\left\langle A_\ve u,\vp 
\right\rangle_{1,2}&:=&
\displaystyle\int_\Omega a_{ij}^{\al \beta}(\ve,x)
\partial_{x_j}u^\beta(x)\partial_{x_i}\vp^\al(x)dx,\\
\left\langle D g,\vp\right\rangle_{1,2}&:=&
\displaystyle\int_\Omega g_{i}^{\al}(x)
\partial_{x_i}\vp^\al(x)
dx,
\end{array}
\right\}
\mbox{ for all } \vp \in W_0^{1,2}(\Omega;\R^n).
\ee
Because of \reff{mona} and \reff{monAhat} and the Lax-Milgram lemma the operators $\hat A$ and $A_\ve$ are bijective from $W_0^{1,2}(\Omega,\R^n)$ onto
$W^{-1,2}(\Omega,\R^n)$, and
\bee
\label{invest}
\sup_{\ve>0}\|A_\ve^{-1}\|_{{\cal L}(W^{-1,2}(\Omega;\R^n);W_0^{1,2}(\Omega;\R^n))}
<\infty.
\ee

Further, we consider the superposition (or Nemycki)  operator 
%$F \in C^1(C(\overline\Omega;\R^n); L^{p_0}(\Omega;\R^{nN}))$ by
$F: C(\overline\Omega;\R^n) \to L^{p_0}(\Omega;\R^{nN})$ which is defined by
\bee
\label{fdef}
%F \in C^1(C(\overline\Omega;\R^n); L^{p_0}(\Omega;\R^{nN})):\;
\big(F(u)\big)(x):=[f_i^\al(x,u(x))]\mbox{ for almost all } x \in \Omega.
\ee
It is $C^1$-smooth because of assumption \reff{cass}.

By definition, a function $u \in W_0^{1,2}(\Omega;\R^n)
\cap C(\overline\Omega;\R^n)
$ is a weak solution to the boundary value problem~\reff{BVP} if and only if
\bee
\label{abBVP}
A_\ve u+DF(u)=0,
\ee
and  $u \in W_0^{1,2}(\Omega;\R^n)\cap C(\overline\Omega;\R^n)$ is a weak solution to the homogenized boundary value problem \reff{hombvp} if and only if $\hat Au+DF(u)=0$, in particular
\bee
\label{abhomBVP}
\hat Au_0+DF(u_0)=0.
\ee
Moreover, by assumption of Theorem \ref{main} we have that
\bee
\label{unique}
\left\{u \in W_0^{1,2}(\Omega;\R^n)
\cap C(\overline \Omega;\R^n):\;
(\hat A+DF'(u_0))u=0\right\}=\{0\}.
\ee

The following theorem, which 
follows from the H-convergence  assumption \reff{H},
is well-known in homogenization theory for linear elliptic equations and  systems 
%with $L^\infty$-coefficients 
(see, e.g.
\cite[Proposition 3.1]{Blanc},  \cite[Lemma 8.6]{Che}, \cite[Theorem 2.3.2]{Shen}). We prove it for the convenience of the reader. 
\begin{theorem}
\label{Shentheorem}
Let be given $u \in W^{1,2}(\Omega;\R^n)$ and $\phi\in W^{-1,2}(\Omega;\R^n)$ and sequences $\ve_1,\ve_2,\ldots>0$ and $u_1,u_2,\ldots \in 
W^{1,2}(\Omega,\R^n)$ such that
$\ve_l\to 0$ for $l \to \infty$ and
\begin{eqnarray}
\label{conv1}
&&u_l\rightharpoonup u \mbox{ for } l\to \infty \mbox{ weakly in }
W^{1,2}(\Omega;\R^n),\\
\label{conv2}
&&\|A_{\ve_l}u_l-\phi\|_{-1,2}\to 0
\mbox{ for } l\to \infty. 
\end{eqnarray}
Then $\hat A u=\phi$.
\end{theorem}
{\bf Proof } Define $\bar u \in W_0^{1,2}(\Omega;\R^n)$ and, for $l=1,2,\ldots$, $\bar u_l \in W_0^{1,2}(\Omega;\R^n)$ by
$
\hat A\bar u=A_{\ve_l}\bar u_l=\phi.
$
Because of assumption \reff{H} it follows that
\begin{eqnarray}
\label{conv3}
&&\bar u^\al_l\rightharpoonup \bar u^\al
\mbox{ for } l\to \infty
\mbox{ weakly in } W_0^{1,2}(\Omega),\\
\label{conv4}
&&a_{ij}^{\al \beta}(\ve_l,\cdot)
\partial_{x_j}\bar u^\beta_l\rightharpoonup  
\hat a_{ij}^{\al \beta}
\partial_{x_j}\bar u^\beta
\mbox{ for } l\to \infty
\mbox{ weakly in } L^{2}(\Omega).
\end{eqnarray}
Moreover, because of \reff{invest} and  \reff{conv2}  we have that $\|u_l-\bar u_l\|_{1,2}=
\|u_l-A_{\ve_l}^{-1}\phi\|_{1,2}\to 0$.
Hence, \reff{conv3} yields that $u_l\rightharpoonup \bar u$
for $l\to \infty$
weakly in  $W_0^{1,2}(\Omega;\R^n)$. Therefore \reff{conv1} implies that $\bar u=u$, and from \reff{conv4} follows that
$\phi=A_{\ve_l}\bar u_l\rightharpoonup
\hat A\bar u=\hat Au$
for $l \to \infty$ weakly in $W^{-1,2}(\Omega;\R^n)$. 
\qed

\section{An abstract result of implicit function theorem type}
\label{rift}
\setcounter{equation}{0}
\setcounter{theorem}{0}
In this section we formulate and prove a rather abstract result of implicit function theorem type (Theorem \ref{Banachtheorem} below)  for equation \reff{abBVP} which will be used later in order to prove Theorem \ref{main}.

\subsection{Auxiliary function spaces}
\label{spaces}
Let $U$ and $V$ be Banach spaces such that
\begin{eqnarray}
\label{Uass}
&&\mbox{$U$ is continuously embedded into $W_0^{1,2}(\Omega;\R^n)$ and compactly into $C(\overline\Omega;\R^n)$,}\\
\label{Vass}
&&\mbox{$V$ is continuously embedded into $W^{-1,2}(\Omega;\R^n)$}
\end{eqnarray}
%Concerning the restrictions of the operators $\hat A, A_\ve and $F$ on $U$ we suppose 
and that
\begin{eqnarray}
\label{hatAass}
&&\hat A,A_\ve \in \LL(U;V),
\;
D \in \LL(L^{p_0}(\Omega;\R^{nN});V),
\\
\label{Bass}
&&A_\ve \mbox{ are bijective from } U \mbox{ onto }  V,\\
\label{supass}
&&\sup_{\ve>0} \|A_\ve^{-1}\|_{{\cal L}(V;U)}<\infty.
%\label{Fass1}
%&&F\in C^1(C(\overline\Omega;\R^n);L^{p_0}(\Omega;\R^{nN}).
\end{eqnarray}
%Remark that 
%\reff{abBVP}, 
%\reff{hatAass} and \reff{Bass}
%yield that $u \in W_0^{1,2}(\Omega;\R^n) \cap C(\overline\Omega;\R^n)$ is a weak solution to \reff{BVP}, i.e. a solution to \reff{abBVP}, if and only if $u \in U$ and
%\bee
%\label{inverse1}
%u = -A_\ve^{-1}DF(u).
%\ee

Proofs of $\ve$-uniform coercivity estimates for singularly perturbed linear differential operators by assuming the contrary have been used for a long time. see, e.g. \cite[Proposition 4.1]{Casteras}, 
\cite[Lemma 3.4.1]{Cao}, 
\cite[Proposition 1.2]{del}, 
\cite[Proposition 5.1(ii)]{HSak}, 
\cite[Lemma 3.4]{Ishii}, 
\cite[Lemma 1.3]{Magnus1},
\cite[Lemma 1]{Taniguchi},
\cite[Proposition 7.1]{Wei}.
The following lemma is an adaption of that approach
to homogenization with H-converging families of diffusion tensors.

\begin{lemma} 
\label{Rob}
There exists $\ve_0>0$ such that
$$
\inf\left\{\|(A_\ve
+DF'(u_0))u\|_V:\;
\ve \in (0,\ve_0], u \in U, \|u\|_U
= 1\right\}>0.
$$
\end{lemma}
{\bf Proof }  Suppose the contrary. Then there exist sequences $\ve_1,\ve_2,\ldots>0$ and $u_1,u_2,\ldots \in U$ such that 
\bee
\label{ABconv}
\ve_l+\|(A_{\ve_l}
+DF'(u_0))u_l\|_V\to 0
\mbox{ for } l \to \infty,
\ee
but
\bee
\label{bound}
\|u_l\|_U= 1 \mbox{ for all } l.
\ee
Because of \reff{Uass} and \reff{bound}
without loss of generality we may assume that there exist $v_1\in W_0^{1,2}(\Omega;\R^n)$ and $v_2\in C(\overline\Omega;\R^n)$
with
$u_l\rightharpoonup v_1$
weakly in $W_0^{1,2}(\Omega;\R^n)$
and 
$\|u_l-v_2\|_\infty\to 0$ for $l \to \infty$.
For $\vp^\al(x):=\mbox{sgn} (v^\al_1(x)-v^\al_2(x))$ follows that
\begin{eqnarray*}
0&=&\lim_{l \to \infty}\int_\Omega(u^\al_l-v^\al_1)\vp^\al dx\\
&=&\lim_{l \to \infty}\int_\Omega(u_l^\al-v^\al_2)\vp^\al dx+\sum_{\al=1}^N\int_\Omega|v^\al_1- v^\al_2|dx=\sum_{\al=1}^N\int_\Omega|v^\al_1-v^\al_2|dx.
\end{eqnarray*}
Therefore
$v_1=v_2=:v\in W_0^{1,2}(\Omega;\R^n)\cap C(\overline\Omega;\R^n)$ and
\begin{eqnarray}
\label{vweak}
&&\mbox{$u_l\rightharpoonup v$
weakly in $W_0^{1,2}(\Omega;\R^n)$
for $l \to \infty$,}\\
\label{vstrong}
&&\mbox{$\|u_l-v\|_\infty\to 0$
for $l \to \infty$}.
\end{eqnarray}
Moreover, \reff{Uass} and \reff{vstrong} yield that
$
%\label{Fconv}
\|DF'(u_0)(u_l-v)\|_V\to 0$  for $l \to \infty$.
Hence, \reff{Vass} and \reff{ABconv} imply that
\bee
\label{ABconv1}
\|A_{\ve_l}u_l+DF'(u_0)v\|_{-1,2}\to 0
\mbox { for } l \to \infty.
\ee
Because of \reff{vweak}, \reff{ABconv1} and Theorem \ref{Shentheorem} (with $\phi=-DF'(u_0)v$) we get $(\hat A+DF'(u_0))v=0$, and \reff{unique} yields that $v=0$.
Therefore \reff{ABconv} and \reff{vstrong} yield that
$\|A_{\ve_l}u_l\|_V\to 0$
for $l \to \infty$.
But this contradicts to \reff{supass} and \reff{bound}.
\qed\\

The linear operators $A_\ve$ are isomorphisms from $U$ onto $V$ (cf. \reff{Bass}), and the linear operator $DF'(u_0)$ is compact 
from $U$ into $V$ (cf. \reff{fdef} and  \reff{Vass}). Therefore the linear operators $A_\ve+DF'(u_0)$ are Fredholm of index zero from $U$ onto $V$, and Lemma \ref{Rob} yields that
%for $\ve \in (0,\ve_0]$
\bee
\label{ABFiso}
\left.
\begin{array}{l}
A_\ve+DF'(u_0)
%\in {\rm Iso}(U;V) 
\mbox{ are bijective from $U$ onto $V$
for $\ve \in (0,\ve_0]$,} \\
\mbox{and $\sup_{\ve \in (0,\ve_0]}
\|(A_\ve+DF'(u_0))^{-1}\|_{\LL(V;U)}<\infty$.}
\end{array}
\right\}
\ee
Therefore, if $\ve \in (0,\ve_0]$, then  $u \in W_0^{1,2}(\Omega;\R^n) \cap C(\overline\Omega;\R^n)$ is a weak solution to \reff{BVP}, i.e. to \reff{abBVP}, 
if and only if  $u \in U$
%(cf. \reff{inverse1}) 
and 
\bee
\label{fpp}
u=G_\ve(u),
\ee
where the maps $G_\ve:U \to U$
are defined by 
\begin{eqnarray*}
G_\ve(u)&:=&u-(A_\ve+DF'(u_0))^{-1}
(A_\ve u+DF(u))\\
&=&(A_\ve+DF'(u_0))^{-1}D(F'(u_0)u-
F(u)).
\end{eqnarray*}
Unfortunately, $u_0 \notin U$
or $\|A_\ve
u_0+DF(u_0)\|_{V}
 \not\to 0$ for $\ve \to 0$,
in general.
Therefore $u_0$ is not an appropriate starting element for the fixed point iteration $u_{l+1}=G_\ve(u_l)$, in general.
In order to apply Banach's fixed point theorem to \reff{fpp} we use as starting elements  the 
$\ve$-depending approximate solutions
$\uu \in U$ to \reff{abBVP} which are 
defined in \reff{barudef1} below.
%close to $u_0$ in $C(\overline\Omega;\R^n)$.

\subsection{Approximate solutions}
\label{approximate}
For $\ve>0$ we define
$\uu \in U$ by
\bee
\label{barudef1}
\uu:=-A_\ve^{-1}DF(u_0).
\ee

\begin{lemma}
\label{barulemma}
We have $\|\uu-u_0\|_\infty+
\|A_\ve
\bar u_\ve+DF(\bar u_\ve)\|_{V}
 \to 0$ for $\ve \to 0$.
\end{lemma}
{\bf Proof } Take a sequence $\ve_1,\ve_2,\ldots>0$ with $\ve_l \to 0$
for $l \to \infty$.
Because of \reff{supass} and \reff{barudef1}
the sequence $\bar u_{\ve_1},\bar u_{\ve_2},\ldots$ is bounded in $U$.
Hence, as in the proof of Lemma \ref{Rob} one shows that
there exist $v \in W_0^{1,2}(\Omega)\cap C(\overline\Omega)$ and a subsequence
$\bar u_{\ve'_1},\bar u_{\ve'_2},\ldots$
such that 
$$
\bar u_{\ve'_l}\rightharpoonup v \mbox{ weakly in } W_0^{1,2}(\Omega)
\mbox{ and }\|u_{\ve'_l}-v\|_\infty
\to 0 \mbox{ for } l \to \infty.
$$
But \reff{barudef1} yields that $A_{\ve'_l}u_{\ve'_l}+DF(u_0)=0$.
Therefore
we can apply Theorem \ref{Shentheorem} again, which yields
that
$
\hat Av=-DF(u_0)=\hat Au_0.
$
Here we used \reff{abhomBVP}.
Because of \reff{monAhat} it follows that $v=u_0$
i.e. the $L^\infty$-strong limit $u_0$ 
%in $L^\infty(\Omega;\R^n)$ 
of the subsequence does not depend on the choice of the subsequence, i.e.
$\|u_{\ve_l}-u_0\|_\infty \to 0$
for $l \to \infty$.
Therefore we get finally that
$$
\|A_{\ve_l}
\bar u_{\ve_l}+DF(\bar u_{\ve_l})\|_{V}=
\|D(F(\bar u_{\ve_l})-F(u_0))\|_{V} \to 0 \mbox{ for } l \to \infty, 
$$
as needed.
\qed\\

Lemma \ref{barulemma} claims that the functions 
$\bar u_\ve$ satisfy equation 
\reff{abBVP} approximately for $\ve \to 0$, i.e. they 
are  "approximate solutions" 
to \reff{abBVP} with $\ve \approx 0$. 
But do they approximate for $\ve \to 0$
exact solutions to \reff{abBVP}, and if yes, in which sense?

\subsection{Application of Banach's fixed point theorem}
\label{Banach}
Roughly speaking, the Theorem \ref{Banachtheorem} below claims the following: 
%If one is able to construct a family $\uu \in U$ which for $\ve \to 0$ solves \reff{abBVP} approximately (such families often are called families of approximate 
%solutions to \reff{abBVP}) and if $\|\uu-u_0\|_\infty \to 0$ for $\ve \to 0$,
%then for $\ve \approx 0$ 
For $\ve \approx 0$ there
exists exactly one solution $u=u_\ve$ to \reff{abBVP} with 
$\|u-u_0\|_\infty \approx 0$, and
$\|u_\ve-u_0\|_\infty$ can be estimated 
by $\|\uu-u_0\|_\infty$ and by the discrepancy $\|A_\ve \uu+DF(\uu)\|_V$ of the approximate solutions
$\uu$ to \reff{abBVP}.

%For $r>0$ and $\ve>0$ denote ${\cal B}_r(\bar u_\ve):=\{u \in U: \|u-\bar u_\ve\|_U \le r\}$.

\begin{theorem}
\label{Banachtheorem}
There exist $\ve_1 \in (0,\ve_0]$ and $\delta>0$ such that for all $\ve \in (0,\ve_1]$ there exists exactly one fixed point $u=u_\ve$ of $G_\ve$ with $\|u-u_0\|_\infty \le \delta$.  Moreover,
\bee
\label{fixest}
\|u_\ve-u_0\|_\infty=O(\|\uu-u_0\|_\infty+\|A_\ve\bar u_\ve+DF(\bar u_\ve)\|_V) \mbox{ for } \ve \to 0.
\ee
\end{theorem}
{\bf Proof }
For $r>0$ and $\ve>0$ denote ${\cal B}_r(\bar u_\ve):=\{u \in U: \|u-\bar u_\ve\|_U \le r\}$.

First, let us show that $G_\ve$ is strictly
contractive with respect to $\|\cdot\|_U$
on ${\cal B}_r(\bar u_\ve)$
if $r$ and $\ve$ are sufficiently small. Take $u_1,u_2 \in U$. Because of \reff{ABFiso} we have
\begin{eqnarray*}
&&\|G_\ve(u_1)-G_\ve(u_1)\|_U
=\left\|(A_\ve+DF'(u_0))^{-1}
D\Big(F'(u_0)(u_1-u_2)-
F(u_1)-F(u_2)\Big)\right\|_U\
\\
&&=\left\|(A_\ve+DF'(u_0))^{-1}
\int_0^1D\Big(F'(u_0)-
F'(su_1-(1-s)u_2)\Big)(u_1-u_2)ds\right\|_U\\
&&\le \mbox{const} \max_{0\le s \le 1}\|F'(u_0)-
F'(su_1-(1-s)u_2)\|_{{\cal L}(C(\overline\Omega;\R^n);L^{p_0}(\Omega;\R^{nN})}\|u_1-u_2\|_U,
\end{eqnarray*}
where the constant does not depend on $\ve$, $u_1$ and $u_2$.
Moreover, if $u_1,u_2 \in{\cal B}_r(\bar u_\ve)$, then 
$$
\|u_0-su_1+(1-s)u_2\|_\infty\le 
\|u_0-\uu\|_\infty+
r  \to 0 \mbox{ for } r,\ve \to 0
\mbox{ uniformly wrt } s \in [0,1].
$$
Therefore the continuity of $F'$ 
from $C(\overline\Omega;\R^n)$
into
${\cal L}(C(\overline\Omega;\R^n);L^{p_0}(\Omega;\R^{nN})$
yields that there exist $r_0>0$ and $\ve_1 \in (0,\ve_0]$ such that for all $\ve \in (0,\ve_1]$
$$
\|G_\ve(u_1)-G_\ve(u_1)\|_U
\le \frac{1}{2}\|u_1-u_2\|_U \mbox{ for } u_1,u_2 \in {\cal B}_{r_0}(\bar u_\ve).
$$

Second, we show that 
$G_\ve$ maps 
${\cal B}_{r_0}(\bar u_\ve)$ into itself for all $\ve \in (0,\ve_1]$
if $\ve_1$ is taken sufficiently small: Indeed, for $u \in {\cal B}_{r_0}(\bar u_\ve)$ we have
\begin{eqnarray*}
&&\|G_\ve(u)-\bar u_\ve\|_U
\le \|G_\ve(u)-G_\ve(\bar u_\ve)\|_U+\|G_\ve(\bar u_\ve)-\bar u_\ve
\|_U\\
&&\le \frac{r_0}{2}+\|(A_\ve+
DF'(u_0))^{-1}(A_\ve\bar u_\ve+DF(\bar u_\ve))\|_U\\
&&
\le \frac{r_0}{2}+\mbox{const}\|A_\ve \uu+DF(\bar u_\ve)\|_V,
\end{eqnarray*}
where the constant doe not depend on $\ve$ and $u$, and Lemma \ref{barulemma} yields the claim.

Now, Banach's fixed point theorem yields that for all $\ve \in (0,\ve_1]$ there exists exactly one fixed point $u=u_\ve$ of $G_\ve$ in ${\cal B}_{r_0}(\bar u_\ve)$.

Let us prove \reff{fixest}:
As above we have 
\begin{eqnarray*}
&&\|u_\ve-\bar u_\ve\|_U\le \|G_\ve(u_\ve)-G_\ve(\bar u_\ve)\|_U+\|G_\ve(\bar u_\ve)-\bar u_\ve
\|_U\\
&&\le
\frac{1}{2}\|u_\ve-\bar u_\ve\|_U+
\mbox{const}\|A_\ve\uu 
+DF(\bar u_\ve)\|_V,
\end{eqnarray*}
where the constant does not depend on $\ve$.
Therefore
$$
\|u_\ve-\bar u_\ve\|_U=O(\|A_\ve
\uu
+DF(\bar u_\ve)\|_V) \mbox{ for } \ve \to 0.
$$
Hence, \reff{Uass} and 
$\|u_\ve-u_0\|_\infty
\le \|u_\ve-\uu\|_\infty
+\|u_\ve-\bar u_\ve\|_\infty
$
yield the claim.

Finally, let us show that the fixed point of $G_\ve$ is unique not only for $\ve \approx 0$ and $\|u-\uu\|_U \approx 0$, but even for $\ve \approx 0$ and $\|u-u_0\|_\infty \approx 0$:
Take $\ve \in (0,\ve_1]$ and $u \in U$ with $G_\ve(u)=u$. Then
\begin{eqnarray*}
0&=&A_\ve u+DF(u)\\
&=&A_\ve \bar u_\ve+DF(\bar u_\ve)+(A_\ve+DF'(u_0)) (u-\bar u_\ve)\\
&&+\int_0^1D\Big(F'(su+(1-s)\bar u_\ve)-F'(u_0)\Big)(u-\bar u_\ve)ds.
\end{eqnarray*}
Therefore \reff{ABFiso} yields that
\begin{eqnarray*}
&&\|u-\bar u_\ve\|_U\\
&&\le \mbox{const}\left\|
A_\ve \bar u_\ve+DF(\bar u_\ve)+
\int_0^1D\Big(F'(su+(1-s)\bar u_\ve)-F'(u_0)\Big)(u-\bar u_\ve)ds\right\|_V\\
&&\le \mbox{const}\Big(\|
A_\ve \bar u_\ve+DF(\bar u_\ve)\|_V
+\|u-\uu\|_\infty\Big),
%&&\;\;\;\,+
%\max_{0\le s \le 1}\|
%F'(su-(1-s))-F'(u_0)\|_{{\cal L}(C(\overline\Omega;\R^n);L^{p_0}(\Omega,\R^{nN})}\|u-\uu\|_\infty,
\end{eqnarray*}
where the constants do not depend on $\ve$ and $u$. Hence, $\|u-\bar u_\ve\|_U$ is small if $\ve$ and $\|u-\uu\|_\infty$ are small
(because of Lemma~\ref{barulemma}). 
But $\ve$ and $\|u-\uu\|_\infty$ are small if $\ve$ and $\|u-u_0\|_\infty$ are small (again because of Lemma~\ref{barulemma}).
%(because of \reff{fixest}).
Therefore the fixed point $u$ is unique if $\ve$ and $\|u-u_0\|_\infty$ are small.
%Here we used Lemma~\ref{barulemma} and \reff{fixest}.
\qed

\begin{remark}
For other applications of Banach or Newton-Kantorowich fixed point iterations near approximate solutions to semilinear elliptic PDEs see \cite[Theorem 3.1]{Breden}
and \cite{Cadiot}.
\end{remark}

\section{Choice of the Banach spaces \boldmath$U$\unboldmath\, 
and \boldmath$V$\unboldmath \, in the case \reff{zwei}}
\label{choicezwei}
\setcounter{equation}{0}
\setcounter{theorem}{0}
As usual, the norm in the Sobolev space $W^{1,p}(\Omega;\R^n)$
%(with $p\in [1,\infty)$) 
is denoted by 
$$
\|u\|_{1,p}:=
\left(\sum_{\al=1}^n
\int_\Omega
\left(|u^\al(x)|^p+
\sum_{i=1}^N 
|\partial_{x_i}u^\al(x)|^p\right)
dx\right)^{1/p}.
$$
The subspace $W_0^{1,p}(\Omega;\R^n)$ 
of $W^{1,p}(\Omega;\R^n)$ 
is the closure with respect to this norm 
of the set 
of all $C^{\infty}$-maps $u:\Omega \to \R^n$ with compact support,
and $W^{-1,p}(\Omega;\R^n):=W_0^{1,p'}(\Omega;\R^n)^*$ is the dual space to 
$W_0^{1,p'}(\Omega;\R^n)$ with 
$1/p+1/p'=1$
and with dual space norm
$$
\|\phi\|_{-1,p}:=\sup\left\{\langle \phi,\vp\rangle_{1,p'}:\; \vp \in W_0^{1,p'}(\Omega;\R^n),
\|\vp\|_{1,p'}\le1\right\},
$$
where
$\langle \cdot,\cdot\rangle_{1,p'}:
W^{-1,p}(\Omega;\R^n)\times
W_0^{1,p'}(\Omega;\R^n)\to \R$
is the dual pairing. 

In order to prove Theorem \ref{main} in the case \reff{zwei} we suppose
\reff{Omass}--\reff{H} and that $N=2$, and we use the results of Section \ref{rift} with
\bee
\label{UVdef1}
U:=W_0^{1,p}(\Omega;\R^n)
\mbox{ and } V:=W^{-1,p}(\Omega;\R^n),
\ee
where $p\in (2,p_0]$ will be choosen below.
We have to verify the conditions \reff{Uass}--\reff{supass} for this choice.

Conditions \reff{Uass}--\reff{hatAass}
follow from the H\"older inequality, asumption $p>N=2$ and the Rellich embedding theorem.
%Condition \reff{Fass1} is satisfied because 
%the map $F$ is the superposition of the $C^1$-Nemycki operator
%$u \in L^\infty(\Omega;\R^n)\mapsto
%[f_i^\al(\cdot,u(\cdot))]
%\in L^\infty(\Omega;\R^{nN})$
%and of 
%the linear bounded operator
%$D$, which is defined in \reff{ABdef}, is bounded from 
%$L^{p}(\Omega;\R^{nN})$ into
%$W^{-1,p}(\Omega;\R^n)$ for any $p \ge 2$ and because 
%of assumption \reff{cass} and
%of the continuous embedding $U \hookrightarrow C(\overline\Omega;\R^n)$.
%the superposition operator $F$, which is defined in \reff{fdef}, is $C^1$-smooth from 
%$U$ into $L^{p_0}(\Omega;\R^{nN})$.

Further, conditions \reff{Bass} and \reff{supass} follow from assumption  \reff{mona}   
and Theorem \ref{Maxreg1} below, which imply the following:
%K. Gr\"oger's maximal regularity results for elliptic systems with non-smooth data \cite[Theorems 1 and 2 and Remark 14]{G}, which  implies the following:
\begin{theorem}
\label{maxreg1}
There exists $p_1>2$  
such that
for all $\ve>0$ and all $p\in[2,p_1]$
the linear operators $A_\ve$
are bijective from $W_0^{1,p}(\Omega;\R^n)$ onto $W^{-1,p}(\Omega;\R^n)$. 
Moreover,
\bee
\label{Ainvert1}
\sup_{\ve>0}\|A_\ve^{-1}\|_{{\cal L}(W^{-1,p}(\Omega;\R^n);W_0^{1,p}(\Omega;\R^n))}
<\infty
\mbox{ for all } p \in [2,p_1].
\ee
\end{theorem}

Hence, Theoren \ref{main} in the case \reff{zwei} follows from Theorem \ref{Banach} with the choice \reff{UVdef1} with arbitrary
$$
2<p \le \min\{p_0,p_1\}.
$$

\section{Choice of the Banach spaces \boldmath$U$\unboldmath\, 
and \boldmath$V$\unboldmath \, in the case \reff{gzwei}}
\label{choicegzwei}
\setcounter{equation}{0}
\setcounter{theorem}{0}
In this section we will prove Theorem \ref{main}
in the case \reff{gzwei}.

In order to introduce Morrey spaces and Sobolev-Morey spaces (cf., e.g. \cite{BenF,Chen,Gia,Kufner,Sa,Troi}) we denote for $x \in \Omega$ and $r>0$
$$
%\label{Omxrdef}
\Omega_{x,r}:=\{\xi\in \Omega:\; \|\xi-x\|<r\},
\mbox{ where $\|\cdot\|$ is the Euclidean norm in } \R^N.
$$
Take $\la \in [0,N)$.
The Morrey space $L^{2,\la}(\Omega;\R^n)$ is the subspace of the Lebesgue $L^2(\Omega;\R^n)$ with norm
$$
\|u\|_{2,\la}:=\sup_{x\in \Omega, r\in (0,1]}r^{-\la/2}\left(\sum_{\al=1}^n\int_{\Omega_{x,r}}|u^\al(\xi)|^2d\xi\right)^{1/2} 
$$
which is defined by
$
L^{2,\la}(\Omega;\R^n):=\{u\in L^2(\Omega;\R^n):
\|u\|_{2,\la}<\infty\}$. 
The Morrey space $L^{2,\la}(\Omega;\R^n)$
is continuously embedded into the Lebesgue space
$L^2(\Omega;\R^n)$, and 
\bee
\label{pembed}
%\label{LMemb}
L^p(\Omega;\R^n) 
\mbox{ is continuously embedded into }
L^{2,\la}(\Omega;\R^n) \mbox{ for } p \ge 2 \mbox{ and }
\la= N(1-2/p).
\ee

The Sobolev-Morrey space $W^{1,2,\la}(\Omega;\R^n)$ is the subspace of 
the Sobolev space  $W^{1,2}(\Omega;\R^n)$ 
with norm
$$
\|u\|_{1,2,\la}:=\left(\sum_{\al=1}^n\left(\int_\Omega|u^\al(x)|^2dx+\sum_{i=1}^N\|\partial_{x_i}u^\al\|^2_{2,\la}\right)\right)^{1/2},
$$
which is defined by
$W^{1,2,\la}(\Omega;\R^n):=\{u\in W^{1,2}(\Omega;\R^n):\;
\|u\|_{1,2,\la}<\infty\}$.
It is well-known that
\bee
\label{embed}
\left.
\begin{array}{l}
W^{1,2,\la}(\Omega;\R^n)
\mbox{ is compactly embedded into }
C^ {0,\al}(\overline \Omega;\R^n)\\
\displaystyle\mbox{for } N-2<\la< N \mbox{ and } 0 \le \al <1-\frac{N-\la}{2}.
\end{array}
\right\}
\ee
Further, we denote $W_0^{1,2,\la}(\Omega;\R^n):=W^{1,2,\la}(\Omega;\R^n)\cap W_0^{1,2}(\Omega;\R^n)$.

And finally, the Sobolev-Morrey space $W^{-1,2,\la}(\Omega;\R^n)$ (sometimes called 
Sobolev-Morrey space of functionals)
is the subspace of $W^{-1,2}(\Omega;\R^n)=W_0^{1,2}(\Omega;\R^n)^*$ 
with norm 
$$
\|\phi\|_{-1,2,\la}:=\sup\left\{
r^{-\la/2}\langle \phi,\vp\rangle_{1,2}:\;
x\in \Omega,r\in (0,1],\vp \in W_0^{1,2}(\Omega;\R^n),\|\vp\|_{1,2}\le 1, \mbox{supp} \,\vp \subset \Omega_{x,r}\right\},
$$
which is defined by
$W^{-1,2,\la}(\Omega;\R^n):=\{\phi\in W^{-1,2}(\Omega;\R^n):\;
\|\phi\|_{-1,2,\la}<\infty\}$, and it is continuously embedded into $W^{-1,2}(\Omega;\R^n)$.
%Here $\|\cdot\|_{1,2}$ is the norm in the Sobolev space $W^{1,2}(\Omega;\R^n)$, and 
%$\langle\cdot,\cdot\rangle_{1,2}:
%W_0^{1,2}(\Omega;\R^n)\times W^{-1,2}(\Omega;\R^n)\to \R$
%is the dual mapping (as introduced in Subsection \ref{Sobolev}.) 
%for vector functions).
It is well-known (cf. e.g. \cite[Theorem 3.9]{Griep})
that 
\bee
\label{Dprop}
D \mbox{ is bounded from $L^{2,\la}(\Omega;\R^{nN})$ into $W^{-1,2,\la}(\Omega;\R^n)$ for all  $\la \in [0,N)$.}
\ee
%for all $\om \in [0,N)$ 
%and $p\ge 2N/(N-\om)$ 
%the linear map 
%$f=[f_i^\al]
%\in L^\infty(\Omega;\R^{nN})
%\mapsto \phi_f\in W^{-1,2}
%(\Omega;\R^n)$, which is defined in 
%\reff{fdef},
%is well-defined and bounded from $L^{\infty}(\Omega;\R^{nN})$
%\times L^{\infty}(\Omega)$ 
%into $W^{-1,2,\la}(\Omega;\R^n)$
%for all $\la \in [0,N)$.

In order to prove Theorem \ref{main} in the case \reff{gzwei} we suppose
\reff{Omass}--\reff{H}, and  we suppose that the diffusion tensor functions in \reff{BVP}
are triangular, i.e.
$$
a_{ij}^{\al \beta}(\ve,x)=0
\mbox{ for all } \al>\beta \mbox{ and }
\ve>0 \mbox{ and almost all } x \in \Omega,
$$
and we use the results of Section \ref{rift} with
\bee
\label{UVdef2}
U:=W_0^{1,2,\lambda}(\Omega;\R^n)
\mbox{ and } V:=W^{-1,2,\lambda}(\Omega;\R^n),
\ee
where $\lambda>N-2$ will be choosen below.
Again, we have to verify the conditions \reff{Uass}--\reff{supass} for this choice.

Conditions \reff{Uass} and \reff{Vass}
follow from the continuous embeddings 
$L^{2,\la}(\Omega) \hookrightarrow
L^{2}(\Omega)$ and 
$W^{-1,2,\la}(\Omega) \hookrightarrow
W^{-1,2}(\Omega)$
and the compact embedding
\reff{embed}.
Condition
\reff{hatAass} is satisfied 
for $A_\ve$ because of
\begin{eqnarray*}
r^{-\la/2}\left|\int_{\Omega_{x,r}} a_{ij}^{\al \beta}(\ve,\xi)\partial_{x_j}u^\beta(\xi)\partial_{x_i}\vp^\al(\xi) d\xi\right|
&\le& \mbox{const }r^{-\la/2}
\sum_{\beta=1}^n\sum_{j=1}^N
\left(\int_{\Omega_{x,r}} |\partial_{x_j}u^\beta(\xi)|^2 d\xi\right)^{1/2}\\  
&\le& \mbox{const }\|u\|_{1,2,\la}
\end{eqnarray*}
for all $x \in \Omega$, $r>0$, $u \in W^{1,2,\la}(\Omega;\R^n)$ and  $\vp \in W_0^{1,2}(\Omega;\R^n)$ with $\mbox{supp}\, \vp \subset \Omega_{x,r}$ and $\|\vp\|_{1,2} \le 1$, where
the constants do not depend on $x$, $r$, $u$ and $\vp$.
And similarly for $\hat A$.

Condition \reff{hatAass} for $D$ with sufficiently small $\la>N-2$ follows from \reff{Dprop} and the continuous embedding (cf. \reff{pembed})
$$
L^{p_0}(\Omega;\R^{nN}) \hookrightarrow
L^{2,\la_0}(\Omega;\R^{nN})
\mbox{ with } \la_0:=N(1-2/p_0).
$$

Further, conditions \reff{Bass} and \reff{supass} with sufficiently small $\la>N-2$ follow from assumption \reff{mona}  
and Theorem \ref{Maxreg2} below, which imply the following:
%the following theorem, which  
%implies that 
%also conditions \reff{Bass} and \reff{supass} are satisfied. 
%It follows from assumptions \reff{aass}--\reff{mona}, \reff{monb} and \reff{gzwei}, and it 
%is proved in  \cite[Theorem 4.1]{Griep} and
%\cite[Lemma 6.2 and Theorem 6.3]{GR}:
\begin{theorem}
\label{maxreg2}
There exists $\la_1\in (N-2,N)$ 
such that 
for all  $\ve>0$ and
$\la\in[0,\la_1]$ the linear  operators 
$A_\ve$ 
are bijective from $W_0^{1,2,\la}(\Omega;\R^n)$ onto $W^{-1,2,\la}(\Omega;\R^n)$.
Moreover,
$$
\sup_{\ve>0}\|A_\ve^{-1}\|_{{\cal L}(W^{-1,2,\la}(\Omega;\R^n);W_0^{1,2,\la}(\Omega;\R^n))}  
<\infty
\mbox{ for all } \la \in [0,\la_1].
$$
\end{theorem}

Hence, Theorem \ref{main} in the case \reff{gzwei} follows from Theorem \ref{Banach} with the choice \reff{UVdef2} with arbitrary
$$
N-2<\la \le \min\{\la_0,\la_1\}.
$$

\section{Appendix 1: About Meyers and Morrey estimates for solutions to linear elliptic systems}
\label{superreg}
\setcounter{equation}{0}
\setcounter{theorem}{0}
Let $\Omega$ be 
a bounded Lipschitz domain in $\R^N$.
For $r \in (0,1)$ we denote by ${\cal M}_r$ the set of all diffusion tensor functions
$[a_{ij}^{\al \beta}] \in L^\infty(\Omega;\R^{n^2N^2})$ such that
\begin{eqnarray*}
&&\essinf\left\{a^{\al \beta}_{ij}(x)v^\al_i v^\beta_j:\; x\in \Omega,\,v \in \R^{nN},\; v^\al_iv^\al_i=1\right\}>r,\\
&&\esssup\left\{\left|a_{ij}^{\al \beta}(\ve,x)\right|:\, x \in \Omega,\, \al,\beta=1,\ldots,n,\, 
i,j=1,\ldots,N\right\}<1/r,
\end{eqnarray*}
and for $a=[a_{ij}^{\al \beta}] \in {\cal M}_r$ we denote $A_a\in  \LL(W_0^{1,2}(\Omega;\R^n);W^{-1,2}(\Omega;\R^n))$ by
$$
\langle A_a u,\vp\rangle_{1,2}:=\int_\Omega a_{ij}^{\al \beta}(x)\partial_{x_j}u^\beta(x)
\partial_{x_i}\vp^\al(x)dx
\mbox{ for all } u,\vp \in W_0^{1,2}(\Omega;\R^n).
$$

The following result is about maximal Sobolev regularity for linear elliptic systems with non-smooth data in the version of 
K. Gr\"oger, see \cite[Theorems 1 and 2 and Remark 14]{G}:

\begin{theorem}
\label{Maxreg1}
For all $r \in (0,1)$ there exists $p_r>2$ such that for all $a \in {\cal M}_r$ and all $p \in 
[2,p_r]$ the linear operator $A_a$
is an isomorphism from 
$W_0^{1,p}(\Omega;\R^n)$ onto $W^{-1,p}(\Omega;\R^n)$. Moreover,
$$
%\label{Ainverta}
\sup_{a \in {\cal M}_r}
\|A_a^{-1}\|_{\LL(W^{-1,p}
(\Omega;\R^n);W_0^{1,p}(\Omega;\R^n))}
<\infty \mbox{ for all }
p \in 
[2,p_r].
$$
\end{theorem}

\begin{remark}
$W^{1,p}$-estimates of solutions to elliptic boundary value problems
%of the type $\|A_a^{-1}\phi\|_{1,p} \le \mbox{const}\,\|\phi\|_{-1,p}$
%\reff{Ainverta}
often are called Meyers type estimates 
because of the initiating paper \cite{ME1975}
of N.G. Meyers, see also \cite{Ga,M}. For the case of smooth boundaries $\partial \Omega$ see 
\cite[Theorem 4.1]{Ben}. See also \cite{CP} for $\Omega$ being a cube and for continuous diffusion coefficients as well as for certain transmission problems and applications to linear elliptic periodic homogenization.
Meyers type estimates are applied to homogenization problems for elliptic boundary value problems with localized defects also in \cite{BDL}.
\end{remark}

For $r \in (0,1)$ denote by 
${\cal M}^{t}_r(\Omega)$ the set of all triangular elements of 
${\cal M}_r(\Omega)$, i.e.
$$
{\cal M}^{t}_r(\Omega):=\left\{[a_{ij}^{\al \beta}] \in
{\cal M}_r(\Omega):\; 
a_{ij}^{\al \beta}=0 \mbox{ for } \al>\beta\right\}.
$$
In \cite[Theorem 4.1]{Griep} and
\cite[Lemma 6.2 and Theorem 6.3]{GR}
there is proved the following
maximal Sobolev-Morrey regularity result for almost triangular elliptic systems with non-smooth data,
which is similar to Theorem \ref{Maxreg1}:

\begin{theorem}
\label{Maxreg2}
For all $r \in (0,1)$ there exists $\la_r>N-2$ such that for all $a=[a_{ij}^{\al \beta}] \in {\cal M}^t_r$ 
and for all $\la \in 
[0,\la_r]$ the linear operator $A_a$
is an isomorphism from 
$W_0^{1,2,\la}(\Omega;\R^n)$ onto $W^{-1,2,\la}(\Omega;\R^n)$.
Moreover,
$$
\sup_{a \in {\cal M}^{t}_r(\Omega)}\|A_a^{-1}\|_{\LL(W^{-1,2,\la}(\Omega;\R^n);W_0^{1,2,\la}(\Omega;\R^n))}
<\infty
\mbox{ for all } \lambda \in [0,\lambda_r].
$$
\end{theorem}

\begin{remark}
%Take $r \in (0,1)$ and 
If $a=\left[a_{ij}^{\al \beta}\right]  \in {\cal M}^t_r$, then
\bee
\label{rema}
\left.
\begin{array}{l}
\mbox{for all $\al=1,\ldots,n$ the $N\times N$ matrices $\left[a_{ij}^{\al \al}(x)\right]$ (no summation over $\al$)}\\ 
\mbox{are positive definite uniformly with respect to $x \in \Omega$,}
\end{array}
\right\}
\ee
i.e. condition \reff{rema} is weaker than condition $a \in 
{\cal M}^t_r$ (with certain $r \in (0,1)$).
But even this weaker condition implies that
$A_a$ is an isomorphisms from 
$W_0^{1,2,\la}(\Omega;\R^n)$ onto $W^{-1,2,\la}(\Omega;\R^n)$ with appropriate $\la>N-2$.
More exactly, for all $a \in L^\infty(\Omega)^{n^2N^2}$ with $a_{ij}^{\al \beta}=0$ for $\al>\beta$ and with \reff{rema} the linear operators $A_a$ are isomorphisms from 
$W_0^{1,2,\la}(\Omega;\R^n)$ onto $W^{-1,2,\la}(\Omega;\R^n)$ with appropriate $\la>N-2$.
\end{remark}

\begin{remark}
The set of isomorphisms between two Banach spaces is open with respect to the uniform operator norm.
Therefore for all $a_0\in {\cal M}^t_r$ 
and for all $\la \in 
[0,\la_r]$ there exists $\ve_{a_0,\la}>0$  such that for all $a\in {\cal M}_r$ 
with $\|a-a_0\|_\infty <\ve_{a_0,\la}$
the linear operator $A_a$
is an isomorphism from 
$W_0^{1,2,\la}(\Omega;\R^n)$ onto $W^{-1,2,\la}(\Omega;\R^n)$.
In this sense one can say that the linear operator $A_a$ is
an isomorphism from 
$W_0^{1,2,\la}(\Omega;\R^n)$ onto $W^{-1,2,\la}(\Omega;\R^n)$ if the diffusion tensor function $a$ is close to be triangular.
It seems to be an open question if $\ve_{a_0,\la}$ depends on $a_0$ via $r$ only or not.
\end{remark}

\begin{remark}
Consider the boundary value problem \reff{BVP} with $u$-independent $f_{i}^{\al}$, i.e. $[f_{i}^{\al}(x,u)]=[g_i^\al(x)]=:G(x)$ with $G \in L^{p_0}(\Omega)^{nN}$ and $p_0>N$. Suppose \reff{Omass}-\reff{mona}, \reff{H} and \reff{gzwei}. Then Theorem \ref{Maxreg2} yields that $u_\ve :=-A_\ve^{-1}DG \in C(\overline \Omega)^n$. But the homogenized diffusion tensor function $[\hat a_{ij}^{\al \beta}]$ is not triangular, in general. Therefore 
Theorem \ref{Maxreg2} cannot be applied to the homogenized boundary value problem \reff{hombvp}, in general. But, anyway,  as in the proof of Lemma \ref{barulemma} one can show  that $u_0 :=-\hat A^{-1}DG \in C(\overline \Omega)^n$ and
$$
\|u_\ve-u_0\|_\infty=\|(A_\ve^{-1}-
\hat A^{-1})DG\|_\infty \to 0
\mbox{ for } \ve \to 0.
$$
\end{remark}

\section{Appendix 2: About \boldmath$C^1$\unboldmath-smooth superposition operators
\boldmath$C(\overline\Omega) \to L^p(\Omega)$\unboldmath
}
\label{super}
\setcounter{equation}{0}
\setcounter{theorem}{0}
Let $\Omega$ be 
a bounded Lipschitz domain in $\R^N$,
and let be given a function $f:\Omega\times \R \to \R$ 
such that
\bee
\label{car}
\mbox{$f(x,\cdot)$ is differentiable for almost all $x\in \Omega$, and
$f(\cdot,u)$ is measurable for  all $u\in \R$.
} 
\ee
Further, suppose that there exists $p\ge 2$ such that for all compact
$K\subset \R$ there exists $g_K \in L^{p}(\Omega)$ such that
\bee 
\label{pesti}
|f(x,u)|+|\partial_uf(x,u)|
\le g_K(x)
\mbox{ for almost all $x \in \Omega$
and all $u \in K$},
\ee
and
\bee 
\label{pconti}
\left.
\begin{array}{l}
\mbox{for all $\ve>0$ there exists $\delta_{\ve,K}>0$ such that}\\
\mbox{$|f(x,u)-f(x,v)|+|\partial_uf(x,u)-\partial_uf(x,u)|
\le \ve g_K(x)$}\\
\mbox{for almost all $x \in \Omega$
and all $u,v \in K$
with $|u-v|\le \delta_{\ve,K}$.}
\end{array}
\right\}
\ee
\begin{lemma}
\label{superpos}
Suppose \reff{car}--\reff{pconti}. Then
there exists $F \in C^1(C(\overline\Omega);
L^{p}(\Omega))$ such that
\bee
\label{Feq}
(F(u))(x)=f(x,u(x))
\mbox{ for almost all } x \in \Omega \mbox{ and all } u \in C(\overline \Omega).
\ee
\end{lemma}
{\bf Proof } Assumption \reff{car}
yields that $f$ is a Caratheodory function, i.e. $f(x,\cdot)$ is continuous for almost all $x \in \Omega$, and $f(\cdot,x)$ is measurable for all $u \in \R$. Therefore $f(\cdot,u(\cdot))$ is measurable for all measurable $u:\Omega \to \R$
(cf., e.g. \cite[Section 1.2]{Ambro}). 

Take $u \in C(\overline\Omega)$. Then there exists a compact set $K \subset \R$ such that $u(x) \in K$ for almost all $x \in \Omega$. Therefore assumption \reff{pesti} yields that
$$
\int_\Omega|f(x,u(x))|^{p}dx
\le\int_\Omega|g_K(x)|^{p}dx<\infty.
$$
Hence, by \reff{Feq} is defined a map 
$F:C(\overline\Omega)\to L^{p}(\Omega)$.

Similarly one shows that 
$\partial_u f(\cdot,u(\cdot))v(\cdot) \in
L^{p}(\Omega)$
for all $u,v \in C(\overline\Omega)$. In other words: For any $u \in C(\overline\Omega$ there is defined a linear operator $L(u):C(\overline\Omega) \to L^{p}(\Omega)$ by
$$
(L(u)v)(x):=\partial_u f(x,u(x))v(x)
\mbox{ for almost all } x \in \Omega
\mbox{ and all } v \in C(\overline\Omega).
$$
This linear operator is bounded because assumption \reff{pesti} yields that
$$
\int_\Omega|\partial_u f(x,u(x))v(x)|^{p}dx \le \|v\|_\infty^{p}\int_\Omega|g_K(x)|^{p}dx
$$
if the compakt $K \subset \R$ is taken such that $u(x) \in K$ for almost all $x \in \Omega$.

Now, let us show that the map $u \in C(\overline\Omega) \mapsto L(u) \in \LL(C(\overline\Omega);
L^{p}(\Omega))$ is continuous. Take $u \in C(\overline\Omega)$ and a compakt $K \subset \R$ such that $u(x)+v(x) \in K$ for almost all $x \in \Omega$ and all $v \in C(\overline\Omega)$ with $\|v\|_\infty \le 1$.
Then assumption \reff{pconti} imlies that for any $\ve >0$ there exists $\delta_{\ve,K}>0$ such that
\begin{eqnarray*}
&&\int_\Omega\left|[(L(u+v)-L(u))w](x)\right|^{p}dx\\
&&=
\int_\Omega\left|\Big(\partial_uf(x,u(x)+v(x))
-\partial_uf(x,u(x))\Big)w(x)
\right|^{p}dx\le \ve^{p}\|w\|_\infty^{p}
\end{eqnarray*}
for all $u,v,w \in C(\overline\Omega)$ with $\|v\|\le \delta_{\ve,K}$.

And finally, let us show that $F$ is differentiable from $C(\overline\Omega)$ into $L^{p}(\Omega)$ and that
$F'(u)=L(u)$ for all 
$u \in C(\overline\Omega)$.
Again, take $u \in C(\overline\Omega)$ and a compakt $K \subset \R$ such that $u(x)+v(x) \in K$ for almost all $x \in \Omega$ and all $v \in C(\overline\Omega)$ with $\|v\|_\infty \le 1$.
Then assumption \reff{pconti} imlies that for any $\ve >0$ there exists $\delta_{\ve,K}>0$ such that
\begin{eqnarray*}
&&\int_\Omega\left|[F(u+v)-F(u)-L(u)v](x)\right|^{p}dx\\
&&
=\int_\Omega\left|\Big(f(x,u(x)+v(x))
-f(x,u(x))
-\partial_uf(x,u(x))\Big)v(x)
\right|^{p}dx\\
&&=\int_\Omega\left|\int_0^1\Big(\partial_uf(x,u(x)+sv(x))
-\partial_uf(x,u(x))\Big)v(x)ds
\right|^{p}dx\\
&&\le\ve^{p}  \int_\Omega |g_K(x)|^pdx \, \|v\|_\infty^{p},
\end{eqnarray*}
if $\|v\|_\infty\le \delta_{\ve,K}$.
\qed

\section*{Acknowledgements}

The author gratefully acknowledges that
his paper is the result of 
a long-standing mathematical cooperation and friendship with Jens A. Griepentrog.

\section*{Declarations}
{\bf Data availability: } No datasets were generated or analysed during the current study.\\
{\bf Conﬂict of interest: }  The author claims that there are no conﬂicts of interes

\end{document}